\documentclass{amsart}
\usepackage{amsmath,amssymb,color,url,mathtools,wasysym,multicol,float,scalefnt,cancel,array,stmaryrd,tikz}
\usepackage[shortlabels]{enumitem}
\usepackage{amsthm}
\usepackage[mathscr]{euscript}
\usepackage{float}
\usepackage{graphicx}
\usepackage[left=1.5in,right=1.5in,top=1in,bottom=1in]{geometry}

\newtheorem{thm}{Theorem}
\newtheorem{lem}[thm]{Lemma}
\newtheorem{prop}[thm]{Proposition}

\theoremstyle{definition}
\newtheorem{definition}{Definition}
\newtheorem{example}{Example}

\newtheorem*{remark}{Remark}

\restylefloat{table}
\numberwithin{equation}{section}

\setenumerate[0]{label=(\alph*)}

\newcounter{commentcounter}

\newcommand{\R}{\mathbb R}

\newcommand{\N}{\mathbb N}

\newcommand{\Z}{\mathbb Z}

\newcommand{\Conf}{\operatorname{Conf}}
\newcommand{\TC}{\operatorname{TC}}
\newcommand{\GC}{\operatorname{GC}}
\newcommand{\Y}{\operatorname{Y}}

\newcommand{\C}{\operatorname{C}}
\newcommand{\F}{\operatorname{F}}
\newcommand{\cl}{\mathscr{C}}
\newcommand{\xopp}{X_2^{opp}}
\newcommand{\tz}{tikzpicture}
\newcommand{\ssize}{\scriptstyle}
\newcommand{\eb}{\overline{e}}
\newcommand{\bd}{{\mathbf d}}

\newcommand*{\blackcircle}{ % \hspace{-.03cm}
\raisebox{-.32ex}
{\scalebox{1.96}{$\bullet$}}}

\begin{document}

\title{Two robots moving geodesically on a tree}

\author{Donald M. Davis}
\address{D.M. Davis \\ Department of Mathematics, Lehigh University\\Bethlehem, PA 18015, USA}
\email{dmd1@lehigh.edu}
\author{Michael Harrison}
\address{M. Harrison \\ Dept.\ Math.\ Sciences, Carnegie Mellon University, Pittsburgh, PA 15213, USA}
\email{mah5044@gmail.com} 
\author{David Recio-Mitter}
\address{D. Recio-Mitter \\ Department of Mathematics, Lehigh University\\Bethlehem, PA 18015, USA}
\email{d.reciomitter@yahoo.com}
\date{\today}


\begin{abstract} We study the geodesic complexity of the ordered and unordered configuration spaces of graphs in both the $\ell_1$ and $\ell_2$ metrics.  We determine the geodesic complexity of the ordered two-point $\varepsilon$-configuration space of any star graph in both the $\ell_1$ and $\ell_2$ metrics and of the unordered two-point configuration space of any tree in the $\ell_1$ metric, by finding explicit geodesics from any pair to any other pair, and arranging them into a minimal number of continuously-varying families.  In each case the geodesic complexity matches the known value of the topological complexity.
\end{abstract}

\maketitle

\section{Introduction and statement of results}

The topological complexity of a space $X$, introduced almost two decades ago by Farber \cite{Farber}, is a homotopy invariant of $X$ which measures the complexity of the \emph{motion planning problem} on $X$.  For example, if $X$ represents the space of all possible states of a robot arm, then $\TC(X)$ measures the number of rules required to determine a complete algorithm which dictates how the robot arm will move from any given initial state to any given final state.  The formal definition requires the notion of the \emph{free path fibration} $PX \to X \times X$, which sends a path $\gamma : [0,1] \to X$ to the pair $(\gamma(0),\gamma(1))$.  The \emph{(reduced) topological complexity} of $X$ is then defined as the smallest number $k$ for which there exists a decomposition
\[ X \times X = \bigsqcup_{i=0}^k E_i\]
into Euclidean Neighborhood Retracts (ENRs) $E_i$, such that there exist local sections $s_i : E_i \to PX$ of the free path fibration.  The sections $s_i$ are the ``rules'' which specify, for any two points $(a,b) \in E_i \subset X \times X$, a path from $a$ to $b$.  The fact that $s_i$ is a section ensures that the rules vary continuously with $(a,b) \in E_i$; here $PX$ is considered with the compact-open topology.

\emph{Geodesic complexity}, recently introduced by the third author \cite{RM}, is a geometric counterpart to topological complexity defined for metric spaces $(X, g)$.  The geodesic complexity also measures the complexity of the motion planning problem, except that all motions are required to follow minimizing geodesics in $X$.

\begin{definition} Let $(X, g)$ be a metric space.  Let $\gamma : [0,1] \to X$ be a path in $X$ and let $\ell(\gamma)$ denote its length.  We say that $\gamma$ is a \emph{(minimal) geodesic} if $\ell(\gamma) = g(\gamma(0),\gamma(1))$.
\end{definition}

\begin{remark} We frequently drop the word ``minimal,'' but we follow the convention that ``geodesic'' always refers to a minimizing geodesic as defined above.  See \cite{RM} for alternate equivalent definitions and some discussion on terminology in metric vs.\ riemannian geometry.
\end{remark}

The formal definition of geodesic complexity is analogous to that of topological complexity.  Let $GX$ be the subspace of $PX$ consisting of geodesics.  Restricting the free path fibration $PX \to X \times X$ to $GX$ yields a map $\pi : GX \to X \times X$.  We note that the map $\pi$ is no longer a fibration (except in the very special case when it is a homeomorphism), although it is sometimes a branched covering (see \cite{RM}, Section 3).

\begin{definition} The \emph{geodesic complexity} $\GC(X, g)$ of a metric space $(X, g)$ is defined as the
smallest number $k$ for which there exists a decomposition
\[ X \times X = \bigsqcup_{i=0}^k E_i\]
into ENRs $E_i$, such that there exist local sections $s_i : E_i \to GX$ of $\pi$.  We refer to the collection $\left\{ s_i \right\}$ as a \emph{geodesic motion planner} and each $s_i$ as a \emph{geodesic motion planning rule} (GMPR).
\end{definition}

By definition, $\TC(X) \leq \GC(X, g)$ for any metric $g$.  In particular, the topological complexity of a space is a homotopy invariant, hence independent of the metric on $X$, but the geodesic complexity of a space genuinely depends on the metric.  The third author showed in \cite{RM} that on each sphere $S^n$, $n \geq 3$, there exist two metrics with different geodesic complexity.  Specifically, for every $k \in \N$, there exists a metric $g$ on $S^{k+2}$ with $\GC(S^{k+2}, g) - \TC(S^{k+2}) \geq k-1$, so the gap between $\GC$ and $\TC$ may be arbitrarily large.

Besides trivial examples, the geodesic complexity has been computed for only a handful of spaces.  In \cite{RM}, the third author computes the geodesic complexity of the flat $n$-torus and the flat Klein bottle and gives lower bounds for the $\GC$ of the standard $2$-torus and for the boundary of the $3$-cube in $\R^3$ which are larger than the $\TC$ of the respective spaces.  In \cite{DR}, the first and third authors compute the $\GC$ of the $n$-dimensional Klein bottles (see also \cite{Davis}), for which the topological complexity is still unknown except in the case of $n=2$ (the ordinary Klein bottle).  In \cite{D}, the first author computes the $\GC$ of certain configuration spaces of $\R^n$.

Our present goal is to compute the geodesic complexity of configuration spaces of certain graphs.  Configuration spaces of graphs are of central importance in topological robotics, since they model the situation of several robots moving along tracks, as in a warehouse (\cite{G}).  We consider the case of two points, either distinguished or indistinguishable, moving on a graph without colliding.  We always assume that the graph $G$ is a tree, and unless otherwise stated, we assume that $G$ is not homeomorphic to an interval.  We obtain explicit descriptions of the geodesics and optimal GMPRs on the configuration spaces of these graphs.

As a motivational example, let $G$ be the figure-$\Y$ graph (three edges emanating from a single vertex) with its usual path metric $d$, and consider the space $F_\varepsilon$ consisting of pairs of points of $\Y$ which are at least distance $\varepsilon$ apart.  A path in $X$ from $a = (a_1,a_2)$ to $b = (b_1,b_2)$ may be thought of as the motion of two particles $\blackcircle$ (the \emph{first} particle) and $\blacksquare$ (the \emph{second} particle) in $\Y$, beginning at $a_1$ and $a_2$, ending at $b_1$ and $b_2$, and staying at least $\varepsilon$ apart throughout their trajectories (see Figure \ref{fig:expath}).  Here $\varepsilon$ is assumed to be small relative to the lengths of the arms.

\begin{figure}[ht!]
\centerline{
\includegraphics[width=1.7in]{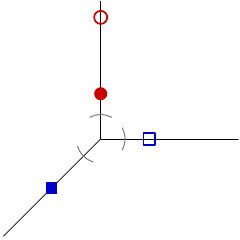} \hspace{.8in} \includegraphics[width=1.7in]{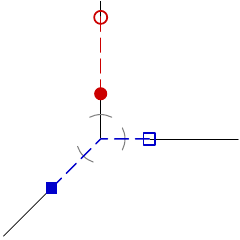}
}
\caption{Left: The $\Y$-graph with an $\varepsilon$-neighborhood of the vertex.  The solid circle and square represent the initial points $a_1$ and $a_2$; the empty circle and square represent the destination points $b_1$ and $b_2$.   Right: A path in $F_\varepsilon$ from $a$ to $b$.}
\label{fig:expath}
\end{figure}

More formally, we define the \emph{ordered two-point $\varepsilon$-configuration space} of a graph $G$ with path metric $d$:
\[
\F_{\varepsilon} \coloneqq F_\varepsilon(G, 2) = \Conf_\varepsilon(G,2) = \left\{ (a_1,a_2) \in G \times G \ \big| \ d(a_1,a_2) \geq \varepsilon \right\},
\]
and the \emph{unordered two-point configuration space} of $G$:
\[
\C \coloneqq C(G,2) = F(G,2) \slash \Z_2 = \left\{ (a_1,a_2) \in G \times G \ \big| \ a_1 \neq a_2 \right\} \slash [(a_1,a_2) \sim (a_2,a_1)].
\]
The lack of $\varepsilon$ in the unordered case will be explained shortly.

The product space $G \times G$ may be endowed with various natural metrics, any of which is inherited by $\F_\varepsilon$ and $\C$.  We focus our attention on the $\ell_1$ and $\ell_2$ metrics.  Writing $a = (a_1,a_2)$ and $b = (b_1,b_2)$ for points $a,b \in G \times G$, the $\ell_1$ and $\ell_2$ metrics on $G \times G$ are defined:
\begin{align*}
\ell_1(a,b) & = |d(a_1,b_1) + d(a_2,b_2)|, \\
\ell_2(a,b) & = \sqrt{d(a_1,b_1)^2 + d(a_2,b_2)^2}.
\end{align*}
In Figure \ref{fig:expath}, it is easy to see that there is a unique geodesic from $a$ to $b$ in the $\ell_2$ metric, obtained when the particles traverse their indicated paths at the appropriate relative speed.

\begin{remark} If $G$ is a tree not homeomorphic to an interval, the ordered two-point configuration space $F(G,2)$ is not geodesically complete in the $\ell_1$ nor $\ell_2$ metric.  To see this, let $v$ be a vertex of degree $\ge3$ with $a_2 = b_2 = v$ and such that $a_1$ and $b_1$ lie on different edges adjacent to $v$, both at equal distance $d$ from $v$. By moving the second point slightly onto a third edge, we can obtain a path in $F(G,2)$ from $(a_1,b_1)$ to $(a_2,b_2)$ of length arbitrarily close to $2d$, but there is no path of length $2d$ in either metric.  Therefore we replace $F(G,2)$ by the geodesically complete space $F_\varepsilon$, which is a ($\Z_2$-equivariant) deformation retract of $F(G,2)$ and hence topologically equivalent to it (see Lemma \ref{lem:defret}).
\end{remark}

\begin{remark} If $G$ is a tree with a vertex $v$ of degree $\ge4$, the unordered two-point configuration space $\C$ is not geodesically complete in the $\ell_2$ metric.  To see this, suppose that $a_1$, $a_2$, $b_1$, and $b_2$ lie on different edges adjacent to $v$, all equidistant from $v$ at distance $\frac{\sqrt{2}}{2}d$.  As above, there is a sequence of paths in $\C$ with length approaching $2d$, but there is no path of length $2d$.  On the other hand, $\C$ is geodesically complete in the $\ell_1$ metric for any tree $G$.  The key difference is that the speed of travel is irrelevant in the $\ell_1$ metric; only the total distance traveled by the particles is important.  Thus there is no penalty for choosing a motion which first keeps one particle fixed while the other moves to its destination, and then moves the remaining particle to its destination.  This strategy eliminates any issues caused by collisions which could potentially occur when both particles move simultaneously.
\end{remark}

Our main results give the values of $\GC(\F_\varepsilon,\ell_i)$, $i \in \left\{1,2\right\}$, and of $\GC(\C,\ell_1)$ for certain graphs $G$.

\begin{thm}
\label{thm:gcl2}
Let $G$ be a star graph, with $k \geq 3$ edges emanating from a single vertex, and let $\F_\varepsilon$ be the ordered two-point $\varepsilon$-configuration space of $G$.  Then
\begin{enumerate}
\item If $k = 3$, $\GC(\F_\varepsilon,\ell_i) = \TC(\F_\varepsilon) = 1$, for $i \in \left\{ 1, 2 \right\}$.
\item If $k \geq 4$, $\GC(\F_\varepsilon,\ell_i) = \TC(\F_\varepsilon) = 2$, for $i \in \left\{ 1, 2 \right\}$.
\end{enumerate}
\end{thm}

\begin{thm}
\label{GCthmun}
Let $G$ be a tree and let $\C$ be the unordered two-point configuration space of $G$.  Then
\begin{enumerate}
\item If $G$ is homeomorphic to an interval, then $\GC(\C,\ell_1) = \TC(\C) = 0$.
\item If $G$ is the $\Y$-graph, $\GC(\C,\ell_1) = \TC(\C) = 1$.
\item Otherwise, $\GC(\C,\ell_1) = \TC(\C) = 2$.
\end{enumerate}
\end{thm}

For a tree $G$, the topological complexities of the configuration spaces $F(G,2)$ and $C(G,2)$ are well-known (see \cite{Far2}); the lower bounds for $\GC$ follow from this and the fact that $F(G,2)$ deformation retracts to $F_\epsilon$ (see Lemma \ref{lem:defret}).  We show the upper bounds by constructing explicit geodesic motion planners. The general strategy for constructing geodesic motion planners on a metric space $(X,g)$ is to analyze the structure of the \emph{total cut locus}, which we define as follows:
\[
\cl = \left\{ (a,b) \in X \times X \ \big| \ \mbox{ there exist multiple geodesic paths from } a \mbox{ to } b\right\}.
\]
On the complement of $\cl$, the geodesic is uniquely determined, and there is a well-defined map $\cl^c \to GX$ which sends a pair of points to the unique geodesic connecting them.  If $X$ is a proper metric space, this map is continuous (see \cite{BridsonHaefliger}, Corollary I.3.13), yielding a GMPR on $\cl^c$.  Thus to construct a geodesic motion planner on a metric space $(X,g)$, it suffices to find GMPRs over the total cut locus $\cl$.

In the $\ell_1$ metric on $\C$, the only pairs of points in $\cl^c$ are those for which there is a path from the starting configuration to the ending configuration in which one particle does not move. Thus $\cl$ contains almost all of $\C \times \C$.  In particular, the rule on $\cl^c$ does not contribute much to the geodesic motion planner, and so the proof of Theorem \ref{GCthmun} still requires partitioning $\C \times \C$ into the appropriate number of ENRs, on each of which there is a continuous choice of a geodesic.  This is the content of Section \ref{unsec}.

On the other hand, an essential part of the proof of the $\ell_2$ case of Theorem \ref{thm:gcl2} is a careful analysis of $\cl$.  To illustrate the difficulties which may arise when determining the total cut locus, we offer a second example, depicted in Figure \ref{fig:exgeodpath}.  Due to the condition that the particles must stay distance $\varepsilon$ apart throughout their trajectories, the second (square) particle in Figure \ref{fig:exgeodpath} must move away from its destination to let the first particle pass.

\begin{figure}[ht!]
\centerline{
\includegraphics[width=1.7in]{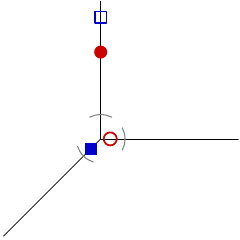} \hspace{.01in} \includegraphics[width=1.7in]{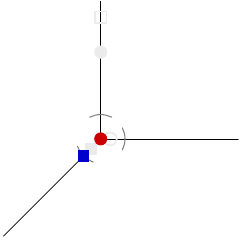} \hspace{.01in} \includegraphics[width=1.7in]{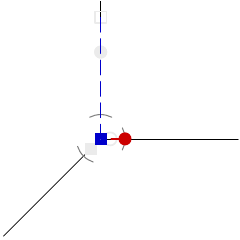} 
}
\caption{With $a$ and $b$ given as in the left image, the geodesic will travel the path which has intermediate stages depicted in the middle and right images.  After the particles arrive at the configuration in the right image, they move directly to their destinations.}
\label{fig:exgeodpath}
\end{figure}

Observe that there is a second feasible path from $a$ to $b$, in which the second particle moves into the right arm of the $\Y$-graph and the first particle moves into the bottom; next the second particle moves through the vertex until both particles can travel directly to their destinations.  Although this path appears almost obviously longer, note that in the limit, taken as the points $a_2$ and $b_1$ approach the vertex, the two paths have equal length.  Thus it is important to carefully determine the geodesics based on the relative locations of $a$ and $b$.  In particular, we will see that most points of the total cut locus require an \emph{orientation switch}, in which an empty arm is used to allow the particles to reconfigure (see Figure \ref{fig:x1x2}, right).  When an orientation switch is necessary, there are two possible considerations: which particle will allow the other to pass, and which arms the particles will use to execute the orientation switch.  These considerations are discussed in greater depth in Section \ref{sec:totalcutlocus}.

\begin{remark} It is interesting to note, in the case such that $G$ is homeomorphic to an interval and $\C$ is considered with the $\ell_1$ metric, that both the total cut locus and its complement are nonempty, yet there exists a single GMPR on the entirety of $\C \times \C$.  This serves as a counterexample to the somewhat intuitive notion that, by definition of $\cl$, a GMPR on $\cl^c$ should not be compatible with one on $\cl$.
\end{remark}

%%%%%%%%%%%%%%%%%%%%%%%%%%%%%%%
%%%%%%%%%%%%%%%%%%%%%%%%%%%%%%%
%%%%%%%%%%%%%%%%%%%%%%%%% ORDERED L2
%%%%%%%%%%%%%%%%%%%%%%%%%%%%%%%
%%%%%%%%%%%%%%%%%%%%%%%%%%%%%%%

\section{The proof of Theorem \ref{thm:gcl2}: Geodesic motion planning on $(\F_\varepsilon, \ell_i)$}

To begin, we exclusively consider the $\ell_2$ metric on $\F_\varepsilon$.  In Section \ref{sec:l1} we make appropriate modifications to the $\ell_2$ case to establish the $\ell_1$ case.  The main difficulty in the proof of the $\ell_2$ case of Theorem \ref{thm:gcl2} is to determine the total cut locus of $\F_\varepsilon$, so that geodesic motion planning rules (GMPRs) can be constructed and the upper bound on $\GC$ can be established.  Before launching into this analysis, we formalize the lower bound in terms of $\TC(F(G,2))$.

%It was inspired by work in \cite{Ab}.
\begin{lem}\label{lem:defret} Let $G$ be a tree and suppose that $\varepsilon > 0$ is not larger than the length of any edge.  Then there is a ($\Z_2$-equivariant) deformation retraction from $F(G,2)$ to $\F_\varepsilon$.
\end{lem}
\begin{proof} Let $(a,b)\in F(G,2)$ with $\bd(a,b)\leq \varepsilon$. If $a$ or $b$ is a vertex $v$, move the other one along its edge to distance $\varepsilon$ from $v$.  If $a$ and $b$ lie on the same edge, move them apart uniformly until they are at distance $\varepsilon$ from one another. If this motion causes one of them to reach a vertex $v$, stop it at the vertex and move the other one to distance $\varepsilon$ from $v$. If a vertex $v$ lies between $a$ and $b$, move them apart with speed proportional to their distances to $v$, until they are at distance $\varepsilon$ from one another.
\end{proof}

The remainder of this section is dedicated to establishing the upper bounds on $\GC$.  We begin by introducing a planar representation of the configuration space $\F_\varepsilon$ which is convenient for depicting certain paths in $\F_\varepsilon$. Consider $(a,b) \in \F_\varepsilon \times \F_\varepsilon$, so that $a_1$, $a_2$, $b_1$, and $b_2$ each lie on some open arm of the $\Y$-graph.  We define
\[
Z \coloneqq Z(a,b) = \left\{ c = (c_1,c_2) \in \F_\varepsilon \ \big| \ c_i \mbox{ is in the same arm as } a_i \mbox{ or } b_i \right\}.
\]
We emphasize that the set $Z$ depends on the points $a$ and $b$, hence will change based on the locations of $a_1$, $a_2$, $b_1$, and $b_2$.

Consider particles $a_1$ and $a_2$ in the top arm of the $\Y$-graph, $b_1$ in the right arm, and $b_2$ in the bottom arm (see Figure \ref{fig:a1a2}).  In this example, $Z$ consists of points $c \in \F_\varepsilon$ such that $c_1$ is in the top or right arm, and $c_2$ is in the top or bottom arm.

The right side of Figure \ref{fig:a1a2} depicts a \emph{representation} of the set $Z$.  The choice of representation is indicated by the directed arcs labeled $x$ and $y$ in the left image.  These arcs define the meaning and orientation of the $x$ and $y$ axes in the right image.  In this case, the $x$-axis \emph{always indicates the position of the first particle}, as follows: the negative $x$-axis represents the negative distance from the vertex to the first particle, assuming that the first particle is in the top arm; the positive $x$-axis represents the distance from the vertex to the first particle, assuming that the first particle is in the right arm.  The negative (positive) $y$-axis is similar; it always indicates distance for the second particle, with respect to the top (bottom) arm.  The interior of the rectangular strip is \emph{forbidden}; for any point inside, the distance between the particles is less than $\varepsilon$, hence the point is not an element of $\F_\varepsilon$.

If the first particle enters the bottom arm, or if the second particle enters the right arm, the configuration of the particles ceases to be an element of $Z$, hence is not representable on the axes shown.  Thus we sometimes refer to elements/subsets of $Z$ as \emph{representable}.

\begin{figure}[ht!]
\centerline{
\includegraphics[width=1.7in]{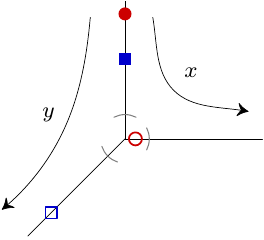} \hspace{.8in} \includegraphics[width=1.7in]{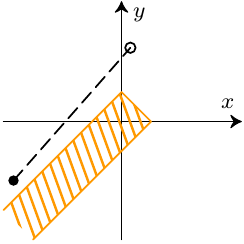}
}
\caption{Left: An element of $\F_\varepsilon \times \F_\varepsilon$ such that $a_1$ and $a_2$ share the top arm; the directed arcs indicate the meaning and orientation of the axes on the right side.  Right: Representations of the initial point $a$ (solid) and destination point $b$ (hollow), along with the unique geodesic connecting them.}
\label{fig:a1a2}
\end{figure}

In the next sections, we rely primarily on geometric intuition to present the GMPRs, and we use the representation only as a visual aid to depict the geodesic between $a$ and $b$.  We show in Section \ref{sec:rep} that the notion of representability can be used to formalize these intuitive statements.  For example, by noting that the map taking $(Z,\ell_2)$ to $(\R^2,\ell_2)$ is an isometry, it is often convenient to argue the uniqueness of geodesics in $\F_\varepsilon$ by using the uniqueness of geodesics in the corresponding representation.  As a related example, in the left side of Figure \ref{fig:a1a2}, it is clear that a geodesic path from $a$ to $b$ has the property that if the particles follow the geodesic, the first particle will not leave the top arm, and the second particle will not enter the top arm.  We tacitly assume such statements in the following sections before verifying them in Section \ref{sec:rep}.

\subsection{Geodesic motion planning on $(\F_\varepsilon, \ell_2)$ for the $\Y$-graph}

We begin our investigation in the case of $k=3$ arms, though the methods here generalize to higher values of $k$.  In this subsection, $G$ always refers to the $\Y$-graph, and $\F_\varepsilon$ always refers to the $2$-point ordered $\varepsilon$-configuration space of $G$, considered with the $\ell_2$ metric.

We define the following partition of $\F_\varepsilon \times \F_\varepsilon$.

\begin{definition}
\label{def:partition} We partition $\F_\varepsilon \times \F_\varepsilon$ into the following subsets.  In the notation below, a subscript indicates the exact number of arms which the points occupy; a superscript indicates whether the orientation of the initial configuration $a$ agrees with (``+'') or disagrees with (``-'') the orientation of the target configuration $b$.  In particular:
\begin{enumerate}
\item $X_1^+$ (resp.\ $X_1^-$) consists of points $(a,b) \in X$ such that $a_1$, $a_2$, $b_1$, $b_2$ all lie on a single arm of $\Y$ (vertex included), and such that the relative orientation of the starting points $a_1$ and $a_2$ \emph{agrees with} (resp.\ \emph{disagrees with}) that of $b_1$ and $b_2$.
\item $X_2^+$ (resp.\ $X_2^-$) consists of points $(a,b) \in X$ such that $a_1$, $a_2$, $b_1$, $b_2$ all lie on exactly two arms of $\Y$ (vertex included), and such that the relative orientation of the starting points $a_1$ and $a_2$ \emph{agrees with} (resp.\ \emph{disagrees with}) that of $b_1$ and $b_2$; see Figure \ref{fig:x1x2}, left (resp.\ right).
\item $X_3$ consists of points $(a,b) \in X$ such that each of the three arms of $\Y$ (vertex excluded) contains at least one of $a_1$, $a_2$, $b_1$, $b_2$ (see Figures \ref{fig:a1a2}, \ref{fig:a1b2}, and \ref{fig:a1b1}).
\end{enumerate}
\end{definition}

\begin{figure}[ht!]
\centerline{
\includegraphics[width=1.7in]{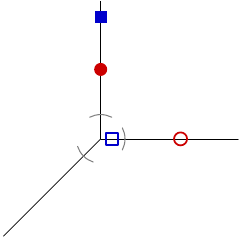} \hspace{.8in} \includegraphics[width=1.7in]{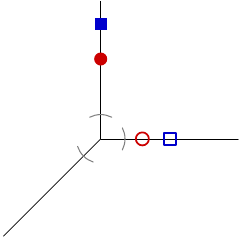}
}
\caption{Left: An example configuration in $X_2^+$; the particles can travel directly to their destination points.  Right: An example configuration in $X_2^-$; the particles must undergo an orientation switch before traveling to their destination points.}
\label{fig:x1x2}
\end{figure}

\begin{prop}
\label{prop:x3}
The sets $X_1^+$, $X_2^+$, and $X_3$ are disjoint from the total cut locus $\cl$.
\end{prop}

\begin{proof}[Proof of Proposition \ref{prop:x3}]  For points in $X_1^+$ and $X_2^+$, there is a unique geodesic taking the direct path from $a$ to $b$.

By the definition of $X_3$, each of the three arms contains a point, hence two points must share one arm.  By utilizing symmetries, we may assume that $a_1$ shares an arm with another point.  Indeed, in the case that $a_2$ shares an arm with $b_1$ (resp.\ $b_2$), the argument follows from the case in which $a_1$ and $b_2$ (resp.\ $b_1$) share an arm, by the $\Z_2$-symmetry swapping $a$ and $b$.  In case the two destination points $b_1$ and $b_2$ share an arm, the argument follows by reversing the geodesic path in the case that $a_1$ and $a_2$ share an arm.  Thus we have three cases to consider: that $a_1$ shares an arm with $a_2$, $b_2$, or $b_1$.

%In each case, the strategy is similar: first, to show that any geodesic between $a$ and $b$ is representable, i.e.\ that the image is a subset of $Z$; second, to show that there exists a unique minimzing geodesic from $a$ to $b$.  In each case, the second item follows easily from the uniqueness of geodesics in the planar representation, and from the fact that the map sending an element of $Z$ to its planar representative is an isometry.  Therefore, given $(a,b) \in Z$, there exists a unique \emph{representable} geodesic $\sigma$ connecting $a$ to $b$, and it remains only to show that every geodesic is representable.

These three cases are depicted with their representations in Figures \ref{fig:a1a2}, \ref{fig:a1b2}, and \ref{fig:a1b1}, respectively.  In each case the particles can move directly to their destinations and do not enter unused arms.  We reiterate that this intuitive proof can be formalized using the notion of representability, first by showing that any geodesic from $a$ to $b$ must be representable (see Lemma \ref{lem:rep}), and then by observing that the representation map $(Z,\ell_2) \to (\R^2,\ell_2)$ is an isometry, so that the unique geodesic between the representations of $a$ and $b$ corresponds to a unique geodesic between $a$ and $b$.
\end{proof}

\begin{figure}[ht!]
\centerline{
\includegraphics[width=1.7in]{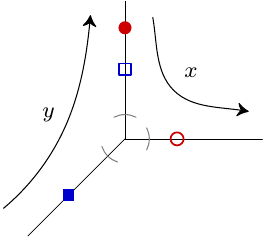} \hspace{.8in} \includegraphics[width=1.7in]{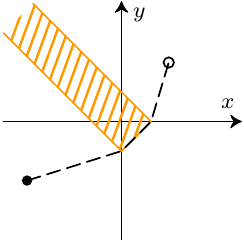}
}
\caption{Left: An element of $X_3$ such that $a_1$ and $b_2$ share the top arm.  Right: Representations of $a$, $b$, and the unique geodesic connecting them.}
\label{fig:a1b2}
\end{figure}

\begin{figure}[ht!]
\centerline{
\includegraphics[width=1.7in]{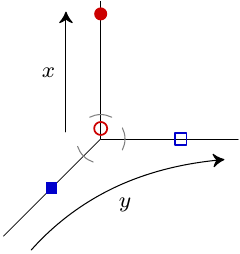} \hspace{1.2in} \includegraphics[width=1in]{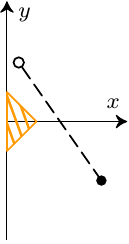}
}
\caption{Left: An element of $X_3$ such that $a_1$ and $b_1$ share the top arm.  Right: Representations of $a$, $b$, and the unique geodesic connecting them.}
\label{fig:a1b1}
\end{figure}

\subsection{The total cut locus of $\F_\varepsilon$}
\label{sec:totalcutlocus} We turn our attention to the total cut locus $\cl$ of $\F_\varepsilon$.  By Proposition \ref{prop:x3}, the total cut locus is contained in $X_1^- \cup X_2^-$.  If $(a,b) \in X_i^-$, $i=1,2$, any path from $a$ to $b$ must undergo an \emph{orientation switch}, in which an empty arm is used to allow the particles to reconfigure; for example, an orientation switch is necessary for the configuration depicted in the right side of Figure \ref{fig:x1x2}.  If $G$ is a star graph with $k > 3$ arms, then any $(a,b)$ which requires an orientation switch is an element of the total cut locus, because one always needs to designate which empty arm is used for the orientation switch.  In the case of the $\Y$ graph with $k = 3$ arms, not all of $X_2^-$ is contained in $\cl$.  In particular, the arm used for the orientation switch is pre-determined by the relative locations of $a$ and $b$, so $(a,b)$ is only in $\cl$ when there is not a preferred particle which moves onto the free arm to let the other pass.  We will formalize these ideas now.

As there exists a GMPR on the complement of the total cut locus, the $\ell_2$ case of Theorem \ref{thm:gcl2}(a) is a consequence of the following:

\begin{prop}
\label{prop:x12d}
There exists a GMPR on the total cut locus $\cl$.
\end{prop}

As $\cl = (X_1^- \cup X_2^-) \cap \cl$, we first consider the spaces $X_1^-$ and $X_2^-$ separately.

\begin{lem} The set $X_1^-$ is a subset of the total cut locus $\cl$, and there exists a GMPR on $X_1^-$.
\label{lem:x1min}
\end{lem}
 
\begin{proof}
We assume that no point is farther from the vertex than $a_1$, keeping in mind the possibility that $a_1 = b_2$.  In case $a_2$ is farthest, there is a symmetry swapping the particles, and if either $b_1$ or $b_2$ is farthest, then $a$ and $b$ may be swapped and the geodesic path may be reversed.

For points in $X_1^-$, an orientation switch is necessary, and there exist exactly two geodesics from $a$ to $b$.  In particular, order the arms clockwise, with arm $1$ at the top, and consider the ordering mod $3$.  If the points all lie on arm $i$, there is a unique geodesic for which the second particle uses arm $j$ for the orientation switch, for each $j \neq i$.

To see that there exist no other geodesics, consider the configuration depicted in Figure \ref{fig:x1d}.  According to the definition of $Z$, the only representable points are those for which both particles are in the top arm.  There are two natural ways to extend the representable region, depicted either in the left or the center of Figure \ref{fig:x1d}.  In either case, the unique geodesic has the representation depicted on the right (although the positive axes have different meaning depending on the choice).  If a path is not representable in either of these two representations, then at some time during the trajectory, both particles lie either in the open right arm or the open bottom arm.  Such a path is non-minimizing (see the displayed statement in the proof of Lemma \ref{lem:rep}).

\begin{figure}[ht!]
\centerline{
\includegraphics[width=1.7in]{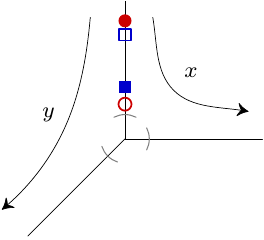} \hspace{.05in} \includegraphics[width=1.7in]{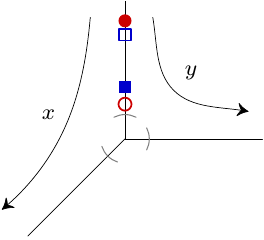} \hspace{.05in} \includegraphics[width=1.7in]{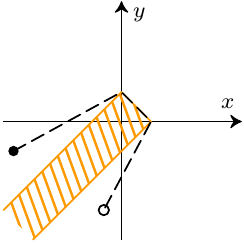}
}
\caption{Left and Center: An element of $X_1^-$ with two possible choices of representation.  Right: The unique geodesic in either representation.}
\label{fig:x1d}
\end{figure}

Now a GMPR on $X_1^-$ can be defined: choose the geodesic for which the second particle uses arm $i+1$ (so that the first particle uses arm $i+2$) for the orientation switch.
\end{proof}

All points of $X_2^-$ require orientation switches, but not all points are in the total cut locus.  Assume momentarily that $a_1$ lies on the top arm, that none of $a_2$, $b_1$, or $b_2$ lies above $a_1$, and that points not on the top arm lie on the bottom arm.  Then there are three possible permutations of the $a_i$ and $b_i$, from top to bottom: $a_1b_2b_1a_2$, $a_1b_2a_2b_1$, and $a_1a_2b_2b_1$ (it is possible that $a_1 = b_2$ or $a_2=b_1$ or $a_2=b_2$).  Other permutations beginning with $a_1$ have agreeing orientation and are not elements of $X_2^-$.

Considering cases dictated by the three possible permutations, as well as the number of $a_i$ and $b_i$ which lie on each arm, we see that there is a unique geodesic in the following cases:
\begin{itemize}
\item if the closed top arm (i.e.\ including the vertex) contains exactly three of the $a_i$ and $b_i$, or if the closed bottom arm contains exactly three of the $a_i$ and $b_i$, then a geodesic is uniquely determined.  See Figure \ref{fig:threeonone} (left and center) for two possible configurations -- other possibilities permute the $a_i$ and $b_i$ but have similar behavior.
\item if $a_1$ and $a_2$ share the open top arm and $b_1$ and $b_2$ share the open bottom arm, there exists a unique geodesic (see Figure \ref{fig:threeonone}, right).
\end{itemize}
Note that in the case of $k > 3$ arms, all such points lie in the total cut locus, because an arm choice is necessary -- see Section \ref{sec:stargraphs} for further discussion.

\begin{figure}[ht!]
\centerline{
\includegraphics[width=1.7in]{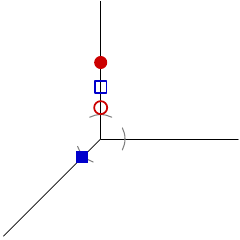} \hspace{.05in} \includegraphics[width=1.7in]{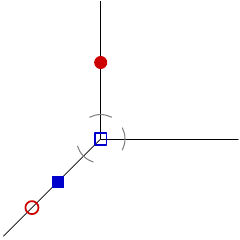} \hspace{.05in} \includegraphics[width=1.7in]{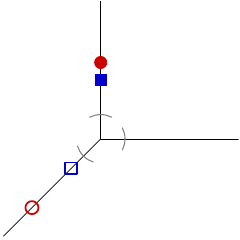}
}
\caption{Three elements of $X_2^-$ in the complement of the total cut locus.  Left: the circle moves into the right arm first.  Center: the square moves into the right arm first.  Right: The square moves into the right arm first.}
\label{fig:threeonone}
\end{figure}

Similar arguments may be made in the cases which do not adhere to our specific assumptions.  For example, if $b_1$ lies above $a_1$, there is a symmetry swapping the particles.  If either $a_2$ or $b_2$ lie above $a_1$, there is a symmetry swapping the initial configuration with the final configuration, and the unique geodesic is the reversed geodesic from the above case.  Finally, the same arguments can be made with respect to any of the three arms.

Thus it remains to consider the case in which $a_1$ and $b_2$ share an open arm, and $a_2$ and $b_1$ share a different open arm.  As such, we define
\[
\xopp = \left\{ (a,b) \in X_2^- \ \big| \ a_i \mbox{ shares an open arm with } b_{i+1} \right\}.
\]
A point of $\xopp$ is depicted in Figure \ref{fig:x2cut}, left.  For $(a,b) \in \xopp$, every path from $a$ to $b$ must undergo an orientation switch, and paths from $a$ to $b$ may be distinguished based on which particle moves through the vertex first for the orientation switch (i.e.\ to enter the empty arm so that the other particle may pass).  We refer to a path such that the $i$th particle passes through the vertex first as a \emph{path of type} $i$.

We claim that for each particle $i$, there exists a unique length-minimizing type $i$ path.  Indeed, any length-minimizing type $1$ path must pass through the point $p \in X$ depicted in Figure \ref{fig:x2cut}, right.  Now $(a,p)$ and $(p,b)$ are in the complement of the total cut locus by Proposition \ref{prop:x3}, so there are unique geodesics connecting $a$ to $p$ and $p$ to $b$.  Their concatenation is the unique length-minimizing type $1$ path from $a$ to $b$.  Similarly, there is a unique length-minimizing type $2$ path from $a$ to $b$.

\begin{figure}[ht!]
\centerline{
\includegraphics[width=1.7in]{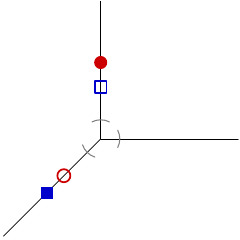} \hspace{.8in} \includegraphics[width=1.7in]{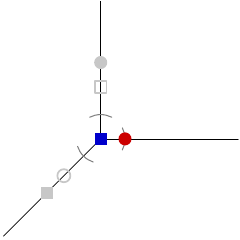}
}
\caption{Left: An element of $\xopp$.  Right:  A length-minimizing type $1$ path from $a$ to $b$ must pass through the depicted point $p$.}
\label{fig:x2cut}
\end{figure}

For at least one value of $i$, the unique length-minimizing type $i$ path is the minimal geodesic from $a$ to $b$.  Occasionally these two paths have equal length, in which case the point $(a,b)$ is in the total cut locus.  Such points are characterized by a certain algebraic condition relating the distances from $a_1$, $a_2$, $b_1$, and $b_2$ to the vertex, but it is not necessary to determine this condition explicitly.  Instead, we simply define
\[
X_2^{eq} = \left\{ (a,b) \in \xopp \ \big| \ \mbox{the unique length-minimizing type } i \mbox{ paths have equal length} \right\},
\]
and $X_2^n = X_2^- - X_2^{eq}$.  We have shown that $X_2^n$ is disjoint from the total cut locus $\cl$, despite the fact that orientation switches are needed for all points of $X_2^-$.  However, with $k > 3$ arms, every orientation switch requires a choice of empty arm to use for the switch, so $X_2^n$ is part of the total cut locus for star graphs (see Section \ref{sec:stargraphs}).

\begin{lem} There exists a GMPR on $X_2^{eq} = X_2^- \cap \cl$.
\label{lem:x2c}
\end{lem}

\begin{proof}
We define the GMPR on $X_2^{eq}$ so that the particle which enters arm $i$ for the orientation switch is the particle which begins on arm $i+1$, i.e.\ the arm adjacent to $i$ in the clockwise direction.
\end{proof}

Finally, we prove Proposition \ref{prop:x12d}, that there exists a GMPR on the total cut locus $\cl$.

\begin{proof}[Proof of Proposition \ref{prop:x12d}]
We have exhibited GMPRs on $X_1^-$ and on $X_2^{eq}$, so it remains to show that these are compatible on the union $\cl$.  In particular, we must check that the GMPRs agree at the points of $X_1^-$ which are limit points of $X_2^{eq} $: these are points such that $a_1 = b_2$ lie at the vertex, or $a_2 = b_1$ lie at the vertex.

\begin{figure}[ht!]
\centerline{
\includegraphics[width=1.7in]{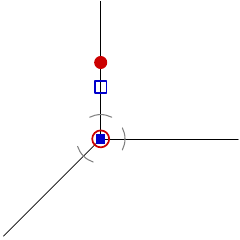} \hspace{.05in} \includegraphics[width=1.7in]{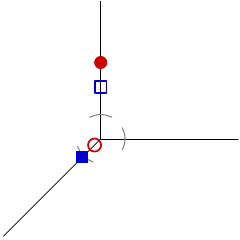} \hspace{.05in} \includegraphics[width=1.7in]{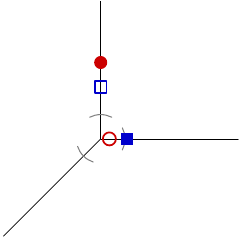}
}
\caption{Left: A limit point $(a,b)$ of $X_2^{eq}$ contained in $X_1^-$.  The GMPR indicates that the second particle enters the right arm, the first enters the bottom, and then both move to their destination.  Center: A possible point of a sequence converging to $(a,b)$.  The GMPR indicates that the second particle enters the right arm, the first enters the bottom, and then both move to their destination.  Right: A possible point of a sequence converging to $(a,b)$.  The GMPR indicates that the second particle stays in the right arm (moving back if necessary), the first particle enters the bottom arm, and then both move to their destination.}
\label{fig:x2dc}
\end{figure}

We consider a fixed limit point $(a,b) \in X_2^{eq}$; without loss of generality, we assume that $a_1$ and $b_2$ lie on the top arm and that $a_2$ and $b_1$ lie on the vertex; see Figure \ref{fig:x2dc}.  Then the GMPR would choose the geodesic for which the second particle moves into the right arm and the first particle moves into the bottom arm to let the second particle pass into the top arm.  Consider a sequence of points $(a^n,b^n) \in X_2^{eq}$ approaching $(a,b)$.  If for some $n$, $a_2^n$ and $b_1^n$ lie in the bottom arm, then the GMPR chooses the geodesic for which the second particle moves into the empty (right) arm and the first moves into the bottom arm.  If for some $n$, $a_2^n$ and $b_1^n$ lie in the right arm, then the GMPR chooses the geodesic for which the first particle enters the empty (bottom) arm (and if necessary, the second particle moves within the right arm to the boundary of the $\varepsilon$-ball around the vertex) to let the second particle pass.  Thus in the limit $(a,b)$, the limiting path is exactly the chosen geodesic at $(a,b)$.
\end{proof}

\begin{remark} To prove Theorem \ref{thm:gcl2}(a) in the $\ell_2$ case we exhibited two GMPRs: one on the total cut locus $\cl$ and one on its complement $\cl^c$.  Note that no GMPR on either $X_1^-$ or $X_2^{eq}$ is compatible with a GMPR on $\cl^c$.
\end{remark}

\subsection{Geodesic motion planning on $(\F_\varepsilon, \ell_2)$ for star graphs}
\label{sec:stargraphs}

We are now equipped to study the geodesic motion planning problem on a star graph $G$, with $k>3$ arms emanating from a single vertex.  The techniques are similar to those for the case $k = 3$, and many of the results from the previous section will be used.

Let $\F_\varepsilon$ be the $2$-point ordered $\varepsilon$-configuration space of $G$.  In addition to the subsets of $\F_\varepsilon \times \F_\varepsilon$ defined in Defintion \ref{def:partition}, we consider the set $X_4$ consisting of points $(a,b) \in \F_\varepsilon \times \F_\varepsilon$ such that there exist four arms of $G$ occupied.  Given $(a,b) \in X_4$, we say that a path from $(a,b)$ is \emph{type} $i$ if the $i$th particle passes through the vertex first.  Recall that type $i$ paths were defined in Section \ref{sec:totalcutlocus} for elements of $X_2^{opp}$; however, for $(a,b) \in X_2^{opp}$, the particle which moves through the vertex first does so to allow the other to pass, whereas for $(a,b) \in X_4$, the particle which moves through the vertex first may travel directly to its destination.

For each particle $i$, there is a unique length-minimizing path among those of type $i$.  At least one of these is a geodesic.  As in the case of $X_2$, we define:
\[
X_4^{eq} = \left\{ (a,b) \in X_4 \ \big| \ \mbox{the unique length-minimizing type } i \mbox{ paths have equal length} \right\},
\]
and $X_4^n = X_4 - X_4^{eq}$ (see Figure \ref{fig:x4un}).  Thus by definition, $X_4^{eq} = X_4 \cap \cl$.

\begin{figure}[ht!]
\centerline{
\includegraphics[width=1.7in]{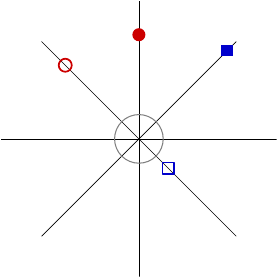} \hspace{.8in} \includegraphics[width=1.7in]{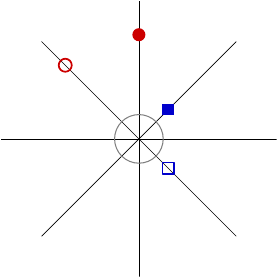}
}
\caption{Left: An element of $X_4^{n}$.  Right: An element of $X_4^{eq}$.}
\label{fig:x4un}
\end{figure}

Recall the discussion prior to Lemma \ref{lem:x2c}: with $k > 3$ arms, $X_2^{n}$ belongs to the total cut locus, since any orientation switch requires a choice of empty arm, and, unlike the situation on the $\Y$-graph,  there are now multiple such choices.

Now the sets $X_4^{n}$ and $X_4^{eq}$, together with $X_3 \cup X_2^+ \cup X_2^{eq} \cup X_2^n \cup X_1^+ \cup X_1^-$, partition $\F_\varepsilon \times \F_\varepsilon$.  The following result establishes three GMPRs on $\F_\varepsilon \times \F_\varepsilon$.

\begin{prop}
\label{prop:propb}
Let $G$ be a star graph with $k$ arms, $k > 3$.  The complement of the total cut locus is given by $X_4^{n} \cup X_3 \cup X_2^+ \cup X_1^+$, and there exist GMPRs on
\begin{enumerate}
\item $X_4^{n} \cup X_3 \cup X_2^+ \cup X_1^+ \cup X_2^{eq}$
\item $X_1^- \cup X_4^{eq}$
\item $X_2^{n}$,
\end{enumerate}
hence $\GC(\F_\varepsilon) \leq 2$.
\end{prop}

We compare to the situation of the $\Y$-graph.  There, the sets $X_4^{n}$ and $X_4^{eq}$ do not exist, the complement of the total cut locus is $X_3 \cup X_2^+ \cup X_1^+ \cup X_2^{n}$, and we exhibited a GMPR on $X_1^- \cup X_2^{eq}$.  With more than three arms, one cannot define compatible GMPRs on $X_1^-$ and $X_2^{eq}$.  So we instead show that the GMPR on $X_2^{eq}$ is compatible with that on the complement of the total cut locus, that $X_4^{eq}$ admits a GMPR compatible with that on $X_1^-$, and that $X_2^{n}$ admits a GMPR.

\begin{proof}
We first remark that $X_4^{n} \cup X_2^+ \cup X_1^+$ are in the complement of the total cut locus by definition, and $X_3$ is in the complement of the total cut locus by Proposition \ref{prop:x3} and because geodesics do not enter unused arms unless an orientation switch is necessary (this can be formalized in the same manner as Lemma \ref{lem:rep}).

\emph{Rule 1. \ } As the complement of the total cut locus, there is a GMPR on $X_4^{n} \cup X_3 \cup X_2^+ \cup X_1^+$. 
The GMPR on $X_2^{eq}$ can be defined in the case of $k > 3$ arms: designate arm $i$ for the orientation switch, where $i$ is the smallest index such that arm $i$ is empty but arm $i+1$ is occupied, and use the particle beginning on arm $i+1$ for the switch.

To show that the GMPR on $X_2^{eq}$ is compatible with that of the complement of the total cut locus, we can show that the closures of the sets are disjoint.

The closure of $X_2^{eq}$ is contained in $X_2^{eq} \cup X_1^-$.  If a limit point $(a,b)$ of $X_4^{n} \cup X_3$ uses only two arms, then one of $a_i$ or $b_i$ must lie at the vertex, hence $(a,b) \notin X_2^{eq}$.  Finally, $X_2^+ \cup X_1^+$ is closed.

\emph{Rule 2. \ } To motion plan on $X_4^{eq}$, observe that $X_4^{eq}$ is a union of finitely many disjoint open sets, defined by indicating which arms contain which particles.  It is enough to motion plan on one of these sets, and the GMPR may be defined simply by always letting the first particle through the vertex first.

The GMPR on $X_1^-$ is defined in Lemma \ref{lem:x1min}; the same argument holds verbatim for star graphs with $k > 3$ arms.

Now $X_1^-$ is closed, and limit points $(a,b)$ of $X_4^{eq}$ occupy at least two arms, so $(a,b) \notin X_1^-$.  Thus the GMPRs are compatible on the union.

\emph{Rule 3. \ } Finally, for points $(a,b) \in X_2^{n}$, the particle which must move first for the orientation switch is pre-determined, and so we only must choose which arm is used for the switch.  We define the GMPR to use the empty arm of lowest index.
\end{proof}

With the $\ell_2$ analysis complete, we now turn our attention to the $\ell_1$ metric on $\F_\varepsilon$.

\subsection{The $\ell_1$ metric on $\F_\varepsilon$ and the proof of Theorem \ref{thm:gcl2}}
\label{sec:l1}
The $\ell_2$-case of Theorem \ref{thm:gcl2} follows from Propositions \ref{prop:x12d} and \ref{prop:propb}, and so it remains to consider the $\ell_1$-case.  To distinguish between geodesics in various metrics, we will use the terminology ``$\ell_i$-geodesic'' to refer to a geodesic in $(\F_\varepsilon, \ell_i)$.

\begin{remark}  Because the $\ell_1$ and $\ell_2$ metrics induce the same topology on $\F_\varepsilon$, the induced compact-open topologies on the path space $P\F_\varepsilon$ are equivalent.  In particular, if $E \subset \F_\varepsilon \times \F_\varepsilon$, and if $s : E \to P\F_\varepsilon$, then $s$ is continuous in the $\ell_1$-induced compact-open topology on $P\F_\varepsilon$ if and only if $s$ is continuous in the $\ell_2$-induced compact-open topology on $P\F_\varepsilon$.  Thus it is not necessary to distinguish notions of continuity when considering geodesics in different metrics.
\end{remark}

To adapt the $\ell_2$ argument to $\ell_1$, the general idea is as follows: we will use essentially the same partition as in the $\ell_2$ case, apart from small modifications to the decompositions of $X_2$ and $X_4$.  In particular, the sets $X_2^{eq}$, $X_2^n$, $X_4^{eq}$, and $X_4^n$, which were previously defined in terms of the $\ell_2$ metric, will be redefined in terms of the $\ell_1$ metric.  The actual GMPRs, which by definition map a point $(a,b)$ to an $\ell_1$-geodesic from $a$ to $b$, will always follow \emph{local} $\ell_2$-geodesics; that is, paths which locally, but perhaps not globally, minimize $\ell_2$-distance.

In particular, we recall that most points $(a,b) \in \F_\varepsilon$ lie in the $\ell_1$ total cut locus: unless one particle stays fixed throughout a minimizing trajectory, one can perturb any $\ell_1$-length-minimizing path, by changing the speed of one of the particles, without changing the $\ell_1$-length.  Even pausing one particle does not penalize the $\ell_1$-length; as long as a motion from $a$ to $b$ does not involve unnecessary backtracking, the motion corresponds to an $\ell_1$-geodesic.  It follows that if an $\ell_2$-geodesic $\gamma$ follows a direct path from $a$ to $b$, as is the case when $(a,b) \in X_1^+ \cup X_2^+$, then $\gamma$ is also an $\ell_1$-geodesic.  More generally, we have the following.

\begin{lem}
\label{lem:l2l1} Let $(a,b) \in X_1 \cup X_2^+ \cup X_3 \cup (X_2^- - \xopp)$ and let $\gamma$ be an $\ell_2$-geodesic from $a$ to $b$.  Then $\gamma$ is an $\ell_1$-geodesic.
\end{lem}

\begin{proof}  For points $(a,b) \in X_1^+ \cup X_2^+$, the uniquely-determined $\ell_2$-geodesic $\gamma$ moves $a$ directly to $b$, hence is an $\ell_1$-geodesic.  For points $(a,b) \in X_3$, $\gamma$ is also uniquely determined, and although such geodesics may contain a brief motion which moves a particle farther from its destination, such motions are strictly necessary.  To see this, observe in Figure \ref{fig:exgeodpath} that any $\ell_1$-geodesic must pass through the points $p$ and $q$ depicted, respectively, in the center and right images, and so the length of any $\ell_1$-geodesic from $a$ to $b$ is equal to the sum $\ell_1(a,p) + \ell_1(p,q) + \ell_1(q,b)$.  The $\ell_2$-geodesic $\gamma$ also passes through $p$ and $q$.  Moreover, the three points $(a,p)$, $(p,q)$, and $(q,b)$ all lie in $X_2^+$, so $\gamma$ is the concatenation of the three uniquely-determined $\ell_2$-geodesics connecting these points, each of which is also an $\ell_1$-geodesic.  Thus $\gamma$ also has $\ell_1$-length $\ell_1(a,p) + \ell_1(p,q) + \ell_1(q,b)$ and hence is an $\ell_1$-geodesic.

Similar arguments apply to points $(a,b) \in X_1^- \cup (X_2^- - \xopp)$.  For such $(a,b)$, there is an $\ell_2$-geodesic from $a$ to $b$ corresponding to each choice of arm(s) used for the orientation switch.  However, once the arm choices are fixed, one may again determine point(s) through which any $\ell_1$-geodesic must pass, and similar arguments may be made.
\end{proof}

Lemma \ref{lem:l2l1} establishes that all $\ell_2$-geodesics are $\ell_1$-geodesics, except perhaps $\ell_2$-geodesics from $a$ to $b$ for points $(a,b) \in \xopp \cup X_4$.  In the next example we give explicit points of $X_2^{opp}$ and of $X_4$ such that Lemma \ref{lem:l2l1} fails.  We recall that for $(a,b) \in \xopp \cup X_4$, a path from $a$ to $b$ is \emph{type} $i$ if the $i$th particle passes through the vertex first.  For $(a,b) \in X_2^{opp}$, and for each particle $i$ and empty arm $j$, there is a unique $\ell_2$-length minimizing path among those of type $i$ using arm $j$ for the orientation switch.  For $(a,b) \in X_4$, no orientation switch is necessary, but for each particle $i$, there is a unique type $i$ $\ell_2$-length minimizing path.

\begin{example} Let $G$ be a star graph with $k > 3$ arms with $\varepsilon = 2$.  Consider $(a,b) \in X_4$,  such that $|a_1| = 1$, $|a_2| = 2$, $|b_1| = 2$, and $|b_2| = 5$ (see Figure \ref{fig:x4x2}, left); here $|\cdot|$ represents distance from the vertex.  Let $\gamma_{i,j}$ represent any path which achieves the minimum $\ell_j$-length among those of type $i$.  Let $d_{i,j}$ represent the length of $\gamma_{i,j}$.  Then
\[
d_{1,1} = 10, \hspace{.2in} d_{2,1} = 12, \hspace{.2in} d_{1,2} = 1 + \sqrt{8} + 5 \approx 8.82, \mbox{\ \ and} \hspace{.2in} d_{2,2} = \sqrt5 + \sqrt8+ \sqrt{13} \approx 8.67.
\]
In particular, the unique $\ell_2$-geodesic $\gamma_{2,2}$ is not an $\ell_1$-geodesic, but the path $\gamma_{1,2}$, which is not an $\ell_2$-geodesic but does minimize $\ell_2$-length among type $1$ paths, is an $\ell_1$-geodesic.  A similar phenomenon occurs for $(a,b) \in \xopp$ with $|a_1| = 1$, $|a_2| = 2$, $|b_1| = 5$, and $|b_2| = 2$ (see Figure \ref{fig:x4x2}, right).
\end{example}
\begin{figure}[ht!]
\centerline{
\includegraphics[width=1.7in]{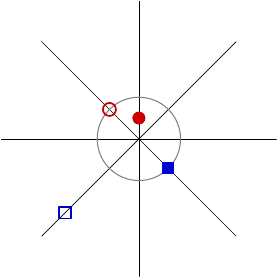} \hspace{.8in} \includegraphics[width=1.7in]{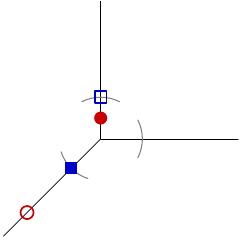}
}
\caption{Elements of $X_4$ (left) and $\xopp$ (right) such that the unique $\ell_2$-geodesic exhibits qualitative behavior distinct from any $\ell_1$-geodesic.}
\label{fig:x4x2}
\end{figure}
We reiterate the intuitive motion planning strategy in the $\ell_1$-case: if $(a,b) \in X_2^{opp} \cup X_4$ has the property that all $\ell_1$-minimizing geodesics are of type $i$, then choose an $\ell_1$-geodesic $\gamma_{i,2}$ (which may not be an $\ell_2$-geodesic).  Now we will formalize these rules, defined on the following sets:
\begin{align*}
(X_2^{eq})_1 & = \left\{ (a,b) \in \xopp \ \big| \ \mbox{ there exist } \ell_1\mbox{-geodesics of types } 1 \mbox{ and } 2 \right\}, \\
(X_4^{eq})_1 & = \left\{ (a,b) \in X_4 \ \  \ \ \big| \ \mbox{ there exist } \ell_1\mbox{-geodesics of types } 1 \mbox{ and } 2 \right\}.
\end{align*}
Note that $(X_2^{eq})_1$ and $(X_4^{eq})_1$ are defined in the same manner as $X_2^{eq}$ and $X_4^{eq}$, with ``$\ell_1$-geodesic'' in place of ``$\ell_2$-geodesic''.  Also analogous to the $\ell_2$ case, we define $(X_2^n)_1 = X_2^- - (X_2^{eq})_1$ and $(X_4^n)_1 = X_4 - (X_4^{eq})_1$.

We define the following GMPRs on these new sets.  We note that, qualitatively, all rules exactly match those in the $\ell_2$ case, as presented in Sections \ref{sec:totalcutlocus} and \ref{sec:stargraphs}; the only difference is that the actual paths are only locally $\ell_2$-length minimizing and may not be $\ell_2$-geodesics.

\begin{lem} There exist GMPRs on the sets $(X_2^{eq})_1$, $(X_2^n)_1$, $(X_4^{eq})_1$, and $(X_4^n)_1$.
\label{lem:gmprl1}
\end{lem}

\begin{proof}  Recall that $\gamma_{i,j}$ represents any path which achieves the minimum $\ell_j$-length among those of type $i$.  For $(a,b) \in X_4$, $\gamma_{i,2}$ is unique.  For $(a,b) \in X_2^-$, we also use the notation $\gamma_{i,2,k}$ to refer to the unique $\ell_2$-length minimizing path from $a$ to $b$ among those type $i$ paths which use leg $k$ for the orientation switch.

On $(X_4^{eq})_1$, there are $\ell_1$-geodesics of type $1$ and $2$; among them we choose $\gamma_{1,2}$.

On $(X_4^n)_1$, there are only $\ell_1$-geodesics of one type $i$; among them we choose $\gamma_{i,2}$.

On $(X_2^{eq})_1$, there are $\ell_1$-geodesics of type $1$ and $2$; among them we choose $\gamma_{i,2,k}$, where $k$ is the smallest index such that arm $k$ is empty but arm $k+1$ is occupied, and $i$ is the particle on arm $k+1$.

On $(X_2^n)_1$, there are only $\ell_1$-geodesics of one type $i$; among them we choose $\gamma_{i,2,k}$, where $k$ is the unused arm of least index.
\end{proof}

When studying compatibility issues in the following proof of Theorem \ref{thm:gcl2}, it is important to note that for $G = \Y$ and $(a,b) \in (X_2^n)_1 - X_2^{opp}$, the path chosen by the above rule is the unique $\ell_2$-geodesic, which is an $\ell_1$-geodesic by Lemma \ref{lem:l2l1}.  Also by Lemma \ref{lem:l2l1}, all $\ell_2$ geodesics on $X_1 \cup X_2^+ \cup X_3$ are $\ell_1$-geodesics, so we define the GMPRs on these sets to use the same geodesics as used in the $\ell_2$ case.

We are now prepared to formalize the proof of Theorem \ref{thm:gcl2} in both the $\ell_1$ and $\ell_2$ cases.

\begin{proof}[Proof of Theorem \ref{thm:gcl2}]
All lower bounds on $\GC$ follow from the known values of $\TC$, together with Lemma \ref{lem:defret}.  The upper bounds for the $\ell_2$ case for parts (a) and (b) follow, respectively, from Propositions \ref{prop:x12d} and \ref{prop:propb}.  In the $\ell_1$ case, it remains to study the compatibility of the GMPRs above.  We emphasize that the rules are qualitatively identical to those in the $\ell_2$ case, so there are only a few compatibility issues to check. \\

(a) Let $G = \Y$.  We claim that there exist two GMPRs, using the same partition as in the $\ell_2$ case, except for the above replacements.  In particular, there are rules on the sets:
\begin{itemize}
\item $X_3 \cup X_2^+ \cup X_1^+ \cup (X_2^n)_1$
\item $X_1^- \cup (X_2^{eq})_1$.
\end{itemize}
The rules on $X_1^-$ and $(X_2^{eq})_1$ have the same qualitative behavior as in the $\ell_2$ case; the verification that these are compatible is the same as in the proof of Proposition \ref{prop:x12d}.

The GMPR on $X_3 \cup X_2^+ \cup X_1^+$ uses the unique $\ell_2$-geodesic and is well-defined by Lemma \ref{lem:l2l1}.  We only must show that the GMPR on $(X_2^n)_1$ is compatible with that on $X_3 \cup X_2^+ \cup X_1^+$.  Note that this was not an issue in the $\ell_2$ case as all sets were in the complement of the $\ell_2$ total cut locus.

To show compatibility, first observe that $X_2^+ \cup X_1^+$ is closed, and limit points of $(X_2^n)_1$ are contained in $(X_2^n)_1 \cup X_1^- \cup (X_2^{eq})_1$.  A limit point $(a,b)$ of $X_3$ may be an element of $(X_2^n)_1$, but it cannot be an element of $X_2^{opp}$ since at least one of $a_1$, $a_2$, $b_1$, or $b_2$ lie at the vertex.  Since the GMPR on $(X_2^n)_1 - X_2^{opp}$ uses the unique $\ell_2$-geodesic from $a$ to $b$ (see the comment below the proof of Lemma \ref{lem:gmprl1}), as does the rule on $X_3$, the rules are compatible. \\

(b) Let $G$ be a star graph with $k > 3$ arms.  We claim that there are rules on the sets:
\begin{itemize}
\item $(X_4^{n})_1 \cup X_3 \cup X_2^+ \cup X_1^+ \cup (X_2^{eq})_1$
\item $X_1^- \cup (X_4^{eq})_1$
\item $(X_2^{n})_1$.
\end{itemize}

Once again, the second and third rules have the same qualitative behavior as in the $\ell_2$ case; the verification that these are compatible is the same as in the proof of Proposition \ref{prop:propb}.

For the first rule, the GMPRs on $(X_4^n)_1$ and $(X_2^{eq})_1$ are defined above, and the GMPR on $X_3 \cup X_2^+ \cup X_1^+$ uses the unique $\ell_2$-geodesic.  Furthermore, $(X_2^{eq})_1$ and $(X_4^n)_1 \cup X_3 \cup X_2^+ \cup X_1^+$ do not share any limit points (analogous to the proof of Proposition \ref{prop:propb}).  Thus we only must show that the GMPR on $(X_4^n)_1$ is compatible with that on $X_3 \cup X_2^+ \cup X_1^+$.  Note that this was not an issue in the $\ell_2$ case as all sets were in the complement of the $\ell_2$ total cut locus.

Suppose that $(a,b) \in X_3 \cup X_2^+$ is a limit point of $(X_4^n)_1$.  We claim that if a point $(a',b') \in (X_4^n)_1$ is sufficiently close to $(a,b)$, then $(a',b')$ is in the complement of the $\ell_2$ total cut locus.  This would establish that the GMPRs at $(a',b')$ and $(a,b)$ both use unique $\ell_2$-geodesics, hence the rules are compatible.

Observe that at least one (and at most two) of $a_1$, $a_2$, $b_1$, and $b_2$ lie at the vertex; we assume $a_1$ is at the vertex, possibly shared with $b_2$.  The GMPR indicates that $(a,b)$ should use the unique $\ell_2$-geodesic $\gamma$ from $a$ to $b$; this geodesic $\gamma$ has the property that the first particle will immediately move down the arm which contains its destination point $b_1$ (there is nothing to be gained by waiting at the vertex nor by moving onto another arm).  Thus $\gamma$ may be considered a type $1$ path, since the first particle moves through the vertex before the second particle.  Any type $2$ path (i.e.\ a path such that the first particle moves onto an empty arm to allow the second particle to go through the vertex ``first'') is strictly longer.  Therefore a sufficiently nearby point $(a',b') \in (X_4^n)_1$ will also have the property that any $\ell_1$- or $\ell_2$-geodesic is type $1$, hence the $\ell_2$-geodesic is unique.
\end{proof}

\subsection{Representations of configurations in $\F_\varepsilon$}
\label{sec:rep}

In previous sections we regularly appealed to intuition regarding the existence and uniqueness of geodesics in various cases.  Here we formalize one statement using the notion of representability; similar arguments may be made for all such statements.

\begin{lem}
\label{lem:rep}
If $(a,b) \in X_3$, then every geodesic from $a$ to $b$ is representable.
\end{lem}

\begin{proof} We reference Figure \ref{fig:a1a2} in the following argument.

Consider any path $\gamma : [0,1] \to \F_\varepsilon$ from $a$ to $b$, and suppose some portion of $\gamma$ is not representable.  In particular, let $(t_1,t_2) \subset [0,1]$ be some interval such that each $\gamma(t_i)$ is representable, but no $\gamma(t), t \in (t_1,t_2)$, is representable.  Thus at each time $t_i$, one of the two particles is at the vertex of the $\Y$-graph, and so $\gamma(t_i)$ is represented on one of the four branches of the axes in the right image of Figure \ref{fig:a1a2}.  We will first show:  
\[
\mbox{If } \gamma(t_1) \mbox{ and } \gamma(t_2) \mbox{ lie on the same axis branch, then } \gamma \big|_{(t_1,t_2)} \mbox{ is not minimizing}.
\]
Indeed, the geodesic in $\F_\varepsilon$ between two points $\gamma(t_1)$ and $\gamma(t_2)$, whose representations lie on the same axis branch, is representable: the representation is an isometry and the geodesic in the representation is the straight line segment along that axis.  Since $\gamma \big|_{(t_1,t_2)}$ is not representable, it is not minimizing.

Now observe that if $\gamma(t_1)$ is on the positive $y$-axis, then $\gamma(t_2)$ also must be; since if the second particle is on the bottom arm when the first particle enters the bottom axis, it must stay there until the first particle leaves.  The same argument applies to the positive $x$-axis.

The remaining option is that one point (we assume it is $\gamma(t_1)$) is represented on the negative $x$-axis, and $\gamma(t_2)$ is represented on the negative $y$-axis.  In this situation, the path $\gamma \big|_{(t_1,t_2)}$ undergoes an orientation switch.  In this case, the second particle enters the right arm while the first particle is on the top arm, then the first particle enters the bottom arm, then the second particle re-enters the top arm, and the first particle moves back to the vertex.  However, this path is exactly the same length as the corresponding path in which the roles of the bottom and right arms are swapped in the previous sentence.  This latter path \emph{is} representable, so we may redefine $\gamma$ on $(t_1,t_2)$ with this representable path, without changing its length.  In this way, every non-representable path is the same length as some representable path.  However, there is a unique representable geodesic connecting $a$ and $b$, and it does not involve any orientation switches, so it is not the same length as any non-representable path.
\end{proof}

Together with the fact that the representation map $(Z,\ell_2) \to (\R^2,\ell_2)$ is an isometry, Lemma \ref{lem:rep} is used to verify that elements of $X_3$ are not in the total cut locus of $\F_\varepsilon$.

When $(a,b) \notin X_3$, it is possible that an orientation switch is necessary, and it is possible that some non-representable path and its representable replacement are actually geodesics.  This was the case for cut locus points in $X_1^-$ and $X_2^-$.  Nevertheless, in these cases the same argument may be used to show that these geodesics are unique among those of certain types; e.g., among geodesics of type $i$ using arm $k$ for an orientation switch.

%%%%%%%%%%%%%%%%%%%%%%%%%%%%%%%
%%%%%%%%%%%%%%%%%%%%%%%%%%%%%%%
%%%%%%%%%%%%%%%%%%%%%%%   UNORDERED L1
%%%%%%%%%%%%%%%%%%%%%%%%%%%%%%%
%%%%%%%%%%%%%%%%%%%%%%%%%%%%%%%

\section{The proof of Theorem \ref{GCthmun}: Geodesic motion planning on $(\C, \ell_1)$}
\label{unsec}

We now turn our attention to the proof of Theorem \ref{GCthmun}.  In case $G$ is homeomorphic to an interval, a single geodesic motion planning rule (GMPR) may be defined explicitly on $\C$.  This establishes part (a) of Theorem \ref{GCthmun}.  For the remainder of this section, $G$ refers to a tree which is not homeomorphic to an interval.

The main accomplishment is to partition $\C \times \C$ into three ENRs (two for the $\Y$ graph), on each of which there is a continuous choice of a geodesic from the initial configuration to the ending configuration.

We begin with a definition and elementary proposition.

\begin{definition} Let $Q$ be a subset of a tree $G$.  The \emph{convex hull} $S(Q)$ of $Q$ is the minimal connected subtree of $G$ containing $Q$; equivalently, it is the union of (the images of) all minimizing geodesics connecting points of $Q$.
\end{definition}

If $x_1,x_2,x_3$ are any three points on a tree $G$, then the image of the geodesic path from $x_2$ to $x_3$ is contained in the union of the paths from $x_1$ to $x_2$ and $x_1$ to $x_3$.  Therefore $S(Q)$ may equivalently be defined as the union of the images of all geodesic paths from one fixed point of $Q$ to all other points of $Q$.  When $Q$ is a set of four points on $G$, $S(Q)$ is the union of three paths emanating from a single point, and so we have the following:

\begin{prop}\label{types} If $Q$ is a set of four points on a tree $G$, then $S(Q)$ must be one of the following five types, pictured in Figure \ref{fig1}.
\begin{itemize}
\item[$Y_1$.] One point of $Q$ is at a vertex, from which paths to the other three are disjoint.
\item[$Y_2$.] There is a vertex from which there are disjoint paths to three of the points of $Q$, and the fourth point of $Q$ lies in the interior of one of these paths.
\item[$X$.] There is a vertex from which paths to all four points are disjoint.
\item[$H$.] There are two vertices, from each of which there are disjoint paths to two points of $Q$, and these four paths are also disjoint from the path between the two vertices.
\item[$I$.] Two points of $Q$ lie inside the path connecting the other two.
\end{itemize}
\end{prop}

\begin{figure}[h!]
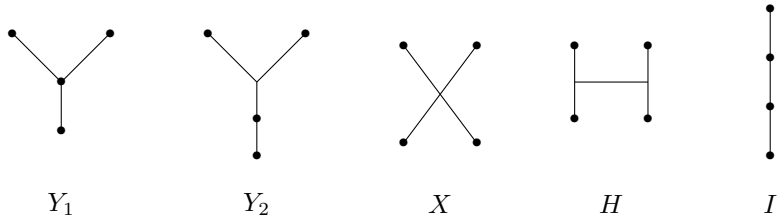

\begin{center}
\begin{\tz}[scale=.65]
\draw (0,0) -- (0,1) -- (-1,2);
\draw (0,1) -- (1,2);
\node at (0,0) {$\ssize\bullet$};
\node at (0,1) {$\ssize\bullet$};
\node at (-1,2) {$\ssize\bullet$};
\node at (1,2) {$\ssize\bullet$};
\draw (4,-.5) -- (4,1) -- (3,2);
\draw (4,1) -- (5,2);
\node at (4,-.5) {$\ssize\bullet$};
\node at (4,.25) {$\ssize\bullet$};
\node at (3,2) {$\ssize\bullet$};
\node at (5,2) {$\ssize\bullet$};
\draw (7,-.25) -- (8.5,1.75);
\draw (8.5,-.25) -- (7,1.75);
\node at (8.5,-.25) {$\ssize\bullet$};
\node at (7,1.75) {$\ssize\bullet$};
\node at (7,-.25) {$\ssize\bullet$};
\node at (8.5,1.75) {$\ssize\bullet$};
\draw (10.5,.25) -- (10.5,1.75);
\draw (12,.25) -- (12,1.75);
\draw (10.5,1) -- (12,1);
\node at (10.5,.25) {$\ssize\bullet$};
\node at (10.5,1.75) {$\ssize\bullet$};
\node at (12,.25) {$\ssize\bullet$};
\node at (12,1.75) {$\ssize\bullet$};
\draw (14.5,-.5) -- (14.5,2.5);
\node at (14.5,-.5) {$\ssize\bullet$};
\node at (14.5,1.5) {$\ssize\bullet$};
\node at (14.5,.5) {$\ssize\bullet$};
\node at (14.5,2.5) {$\ssize\bullet$};
\node at (0,-1.5) {$Y_1$};
\node at (4,-1.5) {$Y_2$};
\node at (7.75,-1.5) {$X$};
\node at (11.25,-1.5) {$H$};
\node at (14.5,-1.5) {$I$};
\end{\tz}
\end{center}
\caption{Types of $S(Q)$}
\label{fig1}
\end{figure}

We number the degree-1 vertices of $G$ from $1$ to $k$ in clockwise order around the tree. In Figure \ref{fig2}, we depict a tree with numbered vertices, and a selection of four points on it of each of our five types.

\begin{figure}[h!]
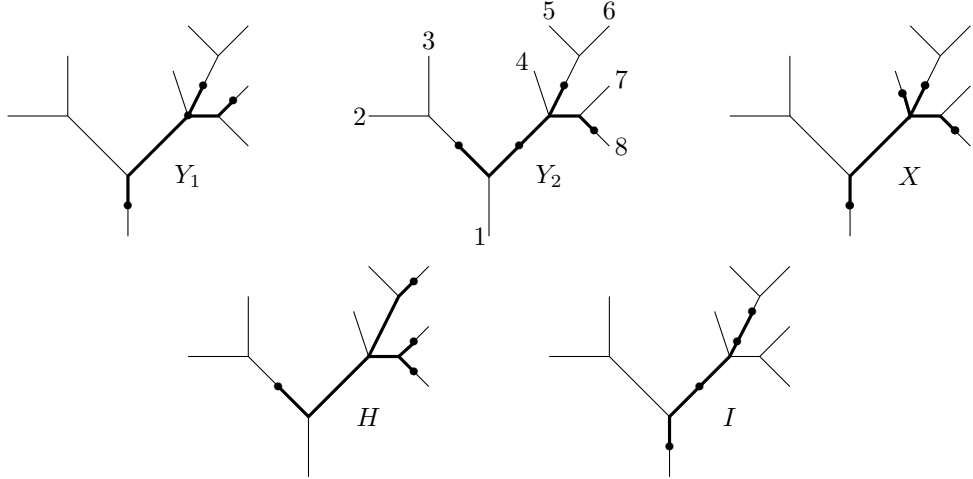

\begin{center}
\begin{\tz}[scale=.4]
\draw (0,0) -- (0,2) -- (-2,4) -- (-4,4);
\draw[very thick] (0,1) -- (0,2) -- (2,4) -- (2.5,5);
\draw[very thick] (2,4) -- (3,4) -- (3.5,4.5);
\draw (-2,4) -- (-2,6);
\draw (0,2) -- (2,4) -- (3,4) -- (4,3);
\draw (3,4) -- (4,5);
\draw (2,4) -- (1.5,5.5);
\draw (2,4) -- (3,6) -- (4,7);
\draw (3,6) -- (2,7);
\draw (12,0) -- (12,2) -- (10,4) -- (8,4);
\draw (10,4) -- (10,6);
\draw (12,2) -- (14,4) -- (15,4) -- (16,3);
\draw (15,4) -- (16,5);
\draw (14,4) -- (13.5,5.5);
\draw (14,4) -- (15,6) -- (16,7);
\draw (15,6) -- (14,7);
\node at (11,3) {$\ssize\bullet$};
\node at (13,3) {$\ssize\bullet$};
\node at (14.5,5) {$\ssize\bullet$};
\node at (15.5, 3.5) {$\ssize\bullet$};
\draw[very thick] (11,3) -- (12,2) -- (14,4) -- (15,4) -- (15.5, 3.5);
\draw[very thick] (14,4) -- (14.5,5);
\draw (24,0) -- (24,2) -- (22,4) -- (20,4);
\draw (22,4) -- (22,6);
\draw (24,2) -- (26,4) -- (27,4) -- (28,3);
\draw (27,4) -- (28,5);
\draw (26,4) -- (25.5,5.5);
\draw (26,4) -- (27,6) -- (28,7);
\draw (27,6) -- (26,7);
\draw[very thick] (24,1) -- (24,2) -- (26,4) -- (27,4) -- (27.5,3.5);
\draw[very thick] (25.8,4.7) -- (26,4) -- (26.5,5);
\node at (24,1) {$\ssize\bullet$};
\node at (25.75,4.75) {$\ssize\bullet$};
\node at (26.5,5) {$\ssize\bullet$};
\node at (27.5,3.5) {$\ssize\bullet$};
\node at (0,1) {$\ssize\bullet$};
\node at (2,4) {$\ssize\bullet$};
\node at (2.5,5) {$\ssize\bullet$};
\node at (3.5,4.5) {$\ssize\bullet$};
\draw (6,-8) -- (6,-6) -- (4,-4) -- (2,-4);
\draw (4,-4) -- (4,-2);
\draw (6,-6) -- (8,-4) -- (9,-4) -- (10,-5);
\draw (9,-4) -- (10,-3);
\draw (8,-4) -- (7.5,-2.5);
\draw (8,-4) -- (9,-2) -- (10,-1);
\draw (9,-2) -- (8,-1);
\draw[very thick] (5,-5) -- (6,-6) --(8,-4) -- (9,-4) -- (9.5,-3.55);
\draw[very thick] (9,-4) -- (9.5,-4.5);
\draw[very thick] (8,-4) -- (9,-2) -- (9.5,-1.5);
\node at (5,-5) {$\ssize\bullet$};
\node at (9.5, -1.5) {$\ssize\bullet$};
\node at (9.5,-4.5) {$\ssize\bullet$};
\node at (9.5,-3.5) {$\ssize\bullet$};
\draw (18,-8) -- (18,-6) -- (16,-4) -- (14,-4);
\draw (16,-4) -- (16,-2);
\draw (18,-6) -- (20,-4) -- (21,-4) -- (22,-5);
\draw (21,-4) -- (22,-3);
\draw (20,-4) -- (19.5,-2.5);
\draw (20,-4) -- (21,-2) -- (22,-1);
\draw (21,-2) -- (20,-1);
\draw[very thick] (18,-7) -- (18,-6) -- (20,-4) -- (20.78,-2.5);
\node at (2,2) {$Y_1$};
\node at (14,2) {$Y_2$};
\node at (26,2) {$X$};
\node at (8,-6) {$H$};
\node at (20,-6) {$I$};
\node at (24,1) {$\ssize\bullet$};
\node at (25.75,4.75) {$\ssize\bullet$};
\node at (26.5,5) {$\ssize\bullet$};
\node at (27.5,3.5) {$\ssize\bullet$};
\node at (18,-7) {$\ssize\bullet$};
\node at (19,-5) {$\ssize\bullet$};
\node at (20.25,-3.5) {$\ssize\bullet$};
\node at (20.75,-2.5) {$\ssize\bullet$};
\node at (11.7,0) {$1$};
\node at (7.7,4) {$2$};
\node at (10,6.5) {$3$};
\node at (13.1,5.8) {$4$};
\node at (14,7.5) {$5$};
\node at (16,7.5) {$6$};
\node at (16.4,5.2) {$7$};
\node at (16.4,3) {$8$};
%\path [fill=red] [thick] (0,1) -- (0,2);
\end{\tz}
\end{center}
\caption{Various $S(Q)$ on a graph.}
\label{fig2}
\end{figure}

\bigskip

If $e$ is an edge emanating from a vertex $v$, let $\eb$ denote the path component of $G-\{v\}$ containing $e$. For example, if $v$ is the degree-4 vertex in the graph in Figure \ref{fig2}, and $e$ is the edge going down and to the left from it, then $\eb$ contains all points between $v$ and the degree-1 vertices labeled 1, 2, and 3. Another example  is when $v$ is the vertex where arms 5 and 6 meet, and $e$ is the edge coming down from it; in this case $\eb$ contains all points between $v$ and the degree-1 vertices labeled 1, 2, 3, 4, 7, and 8. We assign to the edge or to the point on this edge the smallest of these arm numbers. Note that this depends on $v$ as well as $e$.
As an example, if $Q$ consists of points on the arms going out to 5 and 6, a point at the vertex where those arms meet, and a point on the arm going out to 7 and 8, then the arm number for the latter point is 1, not 7.

We will be considering moving from $\{a_1,a_2\}$ to $\{b_1,b_2\}$, with $Q=\{a_1,a_2,b_1,b_2\}$. We will depict $a_1$ and $a_2$ by black dots (B), and $b_1$ and $b_2$ by white (open) dots (W). It is possible that $b_1$ or $b_2$ might coincide with $a_1$ or $a_2$, and so we also consider the possibility that in $Y_2$  the two dots on an edge coincide (and have different colors), and similarly for two adjacent dots on an $I$ diagram.

We subdivide the $I$ diagrams into three types:
\begin{itemize}
\item[$I_1$.] There are one or more vertices of the graph in the interior of the $I$, the endpoint with the smaller arm number contains a white dot, and the endpoint with the larger arm number contains a black dot.
\item[$I_2$.]    There are one or more vertices of the graph in the interior of the $I$, the endpoint with the smaller arm number contains a black dot, and the endpoint with the larger arm number contains a white dot.
\item[$I_3$.] All other $I$ diagrams, so those containing no vertices in the interior and those that do not contain oppositely-colored dots at their endpoints.
\end{itemize}

\begin{prop}\label{3setsun} Let $G$ be a tree.
Partition $\C \times \C$ into three sets as follows. We use the dot color conventions as described above.
\begin{itemize}
\item[$E_1$.] $S(Q)$ is of type $I_1$ or of type $Y_1$ or $Y_2$ such that the arm with smallest arm number contains a black dot,  and, for $Y_2$, the arm with two dots on it contains dots of the same color.
\item[$E_2$.] $S(Q)$ is of type $I_2$ or of type $Y_1$ or $Y_2$ such that the arm with smallest arm number contains a white dot,  and, for $Y_2$, the arm with two dots on it contains dots of the same color.
\item[$E_3$.] Everything else. Thus $S(Q)$ is of type $X$, $H$,  $I_3$, or $Y_2$ such that the arm with two dots on it contains dots of each color.
\end{itemize}

 There is a GMPR on each of $E_1$, $E_2$, and $E_3$.
 \end{prop}

 \begin{remark} {\rm This implies that $\GC(\C, \ell_1)\le 2$, and hence implies Theorem \ref{GCthmun} when $G$ is not the $\Y$ graph, since $\TC(\C)\ge 2$ when $G\ne \Y$.}\end{remark}

 \begin{proof} [Proof of Proposition \ref{3setsun}] Figures \ref{fig3} and \ref{fig4} depict the $Y_2$ cases in $E_1$ and $E_2$, respectively. By placing the dot in the inside of an arm at the vertex, we obtain the $Y_1$ cases of $E_1$ and $E_2$, respectively. The  numbers 1, 2, and 3 indicate the relative order (smallest to largest) of arm numbers associated to the three edges emanating from the vertex $v$. The path that we select is the one that first does uniform motion from the black dot to the white dot indicated by the arrows in the diagram. It is followed immediately by uniform motion from the other black dot to the other white dot. Our uniform motions are always parametrized proportionally to arc length. One can see from the diagrams that collision is avoided. The analogous motion is performed for the associated $Y_1$ diagram. These geodesic paths clearly vary continuously with $(\{a_1,a_2\},\{b_1,b_2\})$.

\begin{figure}
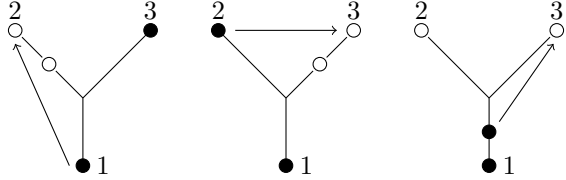

\begin{center}
\begin{\tz}[scale=.45]
\draw (0,-2) -- (0,0) -- (2,2);
\filldraw (0,-2) circle (.2);
\filldraw (2,2) circle (.2);
\draw (0,0) -- (-.8,.8);
\draw (-1.2,1.2) -- (-1.8,1.8);
\draw (-1,1) circle [radius=.2];
\draw (-2,2) circle [radius=.2];
\draw [->] (-.4,-2) -- (-2,1.6);
\node at (.6,-2) {$1$};
\node at (-2,2.6) {$2$};
\node at (2,2.6) {$3$};
\draw (6,-2) -- (6,0);
\draw (6,0) -- (6.8,.8);
\draw (7.2,1.2) -- (7.8,1.8);
\draw (6,0) -- (4,2);
\filldraw (6,-2) circle (.2);
\filldraw (4,2) circle (.2);
\draw (7,1) circle [radius=.2];
\draw (8,2) circle [radius=.2];
\draw [->] (4.5,2) -- (7.5,2);
\node at (6.6,-2) {$1$};
\node at (4,2.6) {$2$};
\node at (8,2.6) {$3$};
\node at (12.6,-2) {$1$};
\node at (10,2.6) {$2$};
\node at (14,2.6) {$3$};
\draw (12,-2) -- (12,0) -- (13.8,1.8);
\draw (12,0) -- (10.2,1.8);
\filldraw (12,-2) circle (.2);
\filldraw (12,-1) circle (.2);
\draw (10,2) circle [radius=.2];
\draw (14,2) circle [radius=.2];
\draw [->] (12.3,-.7) -- (13.9,1.6);
\end{\tz}
\end{center}
\caption{$Y_2$ graphs in $E_1$.}
\label{fig3}
\end{figure}

\begin{figure}
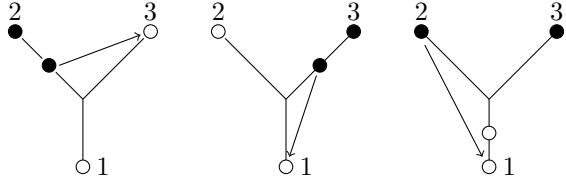

\begin{center}
\begin{\tz}[scale=.45]
\draw (0,-1.8) -- (0,0) -- (1.8,1.8);
\draw (0,0) -- (-.8,.8);
\draw (-1.2,1.2) -- (-1.8,1.8);
\filldraw (-1,1) circle (.2);
\filldraw (-2,2) circle (.2);
\draw (0,-2) circle [radius=.2];
\draw (2,2) circle [radius=.2];
\draw [->] (-.7,1) -- (1.7,1.9);
\node at (.6,-2) {$1$};
\node at (-2,2.6) {$2$};
\node at (2,2.6) {$3$};
\node at (6.6,-2) {$1$};
\node at (4,2.6) {$2$};
\node at (8,2.6) {$3$};
\node at (12.6,-2) {$1$};
\node at (10,2.6) {$2$};
\node at (14,2.6) {$3$};
\draw (6,-1.8) -- (6,0) -- (4.2,1.8);
\draw (6,0) -- (8,2);
\filldraw (7,1) circle (.2);
\filldraw (8,2) circle (.2);
\draw (6,-2) circle [radius=.2];
\draw (4,2) circle [radius=.2];
\draw [->] (6.9,.7) -- (6.1,-1.7);
\draw (12,-1.8) -- (12, -1.2);
\draw (12,-.8) -- (12,0) -- (14,2);
\draw (12,0) -- (10,2);
\filldraw (10,2) circle (.2);
\filldraw (14,2) circle (.2);
\draw (12,-2) circle [radius=.2];
\draw (12,-1) circle [radius=.2];
\draw [->] (10.1,1.6) -- (11.8,-1.7);
\end{\tz}
\end{center}
\caption{$Y_2$ graphs in $E_2$.}
\label{fig4}
\end{figure}

After possibly reorienting, the $I_1$ and $I_2$ diagrams are in the order either BBWW or BWBW, with adjacent B and W possibly at the same position. We choose the path which first moves from the rightmost B uniformly to the rightmost W, followed immediately by uniform motion from the other B to the other W.

For each of the six diagrams of Figures \ref{fig3} and \ref{fig4}, there are two ways that a white dot and black dot could both approach the vertex, having an $I$ diagram of BBWW form as the limiting diagram. In every case, one of those will have as the limiting motion the geodesic just described for the BBWW $I$ diagram, while the other
 will not be in the same $E_i$ set as  the $Y$ diagram that approached it, and so we do not have to worry about compatibility.
 For example, in the first diagram of Figure \ref{fig3}, if the inner  white dot and dot 1 approach the vertex, the limiting $I$ diagram has motion of the first type just described, while if the inner white dot and dot 3 approach the vertex, the limiting $I$ diagram is of the second type because the endpoint with the smaller number has a black dot, as is the case for $I_2$ diagrams. The thing that makes this work is that the Figure \ref{fig3} diagrams have
 the motion between outside dots increasing the arm number ($1\to2$, $2\to 3$, and $1\to 2$, respectively), which corresponds to the motion between outside dots of $I_2$ diagrams. A similar, reversed discussion holds for Figure \ref{fig4} diagrams.

This completes the proof that we have a GMPR on $E_1$ and $E_2$. The argument for $E_3$ which follows is rather similar.

For the $I_3$ diagrams of the form BWWB and WBBW, we use simultaneous uniform motion from the B to the adjacent W. For those of the form BBWW and BWBW (which are all in a single edge), we move from the rightmost B first.

For $X$ diagrams, we first move uniformly  from the black dot with smallest arm number to the white dot with smallest arm number, and then uniformly from the other black dot to the other white dot. As one or two (differently colored) dots approach the vertex, a $Y_1$ or $I_1$ or $I_2$ diagram is obtained. Since these are not in $E_3$, we need not worry about compatibility. A similar rule is used for an $H$ diagram in which the arms emanating from each of the vertices $v_0$ and $v_1$ contain dots of the same color; we first move uniformly from the black dot
with smallest arm number to the white dot with smallest arm number, and then uniformly between the other two dots. The limit as one or two (differently colored) dots approach a vertex is either a $Y_2$ diagram in $E_1$ or $E_2$ or an $I_1$ or $I_2$ diagram, and so again compatibility is not an issue.

For the $Y_2$ diagrams in $E_3$, we can use simultaneous uniform motion between the two dots on one arm, and between the dots on the other two arms, always from black to white, or course. There is no risk of collision. If one of the dots which are unaccompanied on their arm approaches the vertex, we obtain either an $I_3$ diagram with compatible simultaneous motion or an $I_1$ or $I_2$ diagram. If one or both of the dots on the doubly-occupied arm approach the vertex, the limiting diagram is not in $E_3$, so compatibility is not an issue.
Similarly for $H$ diagrams with dots of different colors on each of the two arms emanating from each vertex, move simultaneously and uniformly between the dots on the two arms emanating from each vertex. The limit as one or two of the dots approach their vertex is either a $Y_2$ or $I$ diagram in $E_3$ with compatible motion or else an $I$ diagram not in $E_3$.

This completes the GMPR on $E_3$ and thus completes the proof of the theorem.
\end{proof}

The following result implies Theorem \ref{GCthmun} when $G$ is the $\Y$ graph, since $\TC(\C)\ge1$.
\begin{prop}\label{Yun} If $G$ is the $\Y$ graph, then $\C \times \C$ can be partitioned into two ENRs with a GMPR on each.\end{prop}
\begin{proof} There are no $X$ or $H$ diagrams. We can describe all the diagrams in a rotation-invariant way. The $Y_2$ diagrams and one $I$ diagram are placed into the two sets $E_1'$ and $E_2'$ as suggested in Figure \ref{fig5}. The depiction of contiguous white and black dots means that they can appear in either order or at the same point, and in the last diagram they may appear anywhere between the two end points.

\begin{figure}
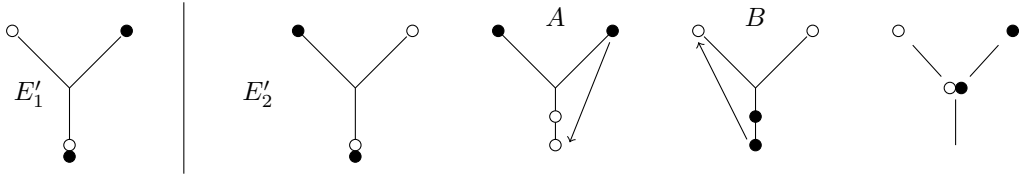

\begin{center}

\begin{\tz}[scale=.38]
\draw (0,-1.8) -- (0,0) -- (1.8,1.8);
\draw (0,0) -- (-1.8,1.8);
\filldraw (2,2) circle (.2);
\filldraw (0,-2.4) circle (.2);
\draw (-2,2) circle [radius=.2];
\draw (0,-2) circle [radius=.2];
\node at (-1.4,-.2) {$E_1'$};
\draw (4,-3) -- (4,3);
\draw (10,-1.8) -- (10,0) -- (8,2);
\draw (10,0) -- (11.8,1.8);
\filldraw (8,2) circle (.2);
\filldraw (10,-2.4) circle (.2);
\draw (12,2) circle [radius=.2];
\draw (10,-2) circle [radius=.2];
\node at (6.6,-.2) {$E_2'$};
\draw (17,-1.8) -- (17,-1.2);
\draw (17,-2) circle [radius=.2];
\draw (17,-1) circle [radius=.2];
\draw (17,-.8) -- (17,0) -- (19,2);
\draw (17,0) -- (15,2);
\filldraw (19,2) circle (.2);
\filldraw (15,2) circle (.2);
\draw [->] (18.9,1.6) -- (17.5,-1.9);
\draw (24,-2) -- (24,0) -- (25.8,1.8);
\draw (24,0) -- (22.2,1.8);
\filldraw (24,-2) circle (.2);
\filldraw (24,-1) circle (.2);
\draw (26,2) circle [radius=.2];
\draw (22,2) circle [radius=.2];
\draw [->] (23.7,-1.8) -- (22,1.6);
\node at (17,2.5) {$A$};
\node at (24,2.5) {$B$};
\draw (31,-2) -- (31,-.4);
\draw (30.8,0) circle [radius=.2];
\filldraw (31.2,0) circle (.2);
\draw (29,2) circle [radius=.2];
\filldraw (33,2) circle (.2);
\draw (30.5,.4) -- (29.5,1.5);
\draw (31.5,.4) -- (32.5,1.5);
\end{\tz}
\end{center}
\caption{ Some $Y_2$ diagrams and an $I$ diagram.}
\label{fig5}
\end{figure}

All $Y_1$ diagrams are placed in $E_2'$, and all $I$ diagrams except the ones of the type illustrated in Figure \ref{fig5} are placed in $E_1'$.
The GMPR is simultaneous uniform linear motion on all the $I$ diagrams and the first two diagrams in Figure \ref{fig5}. For the $I$ diagrams, there is only one way that this can be done without collision, while for the two in Figure \ref{fig5}, it is between the two bottom dots and between the two top dots. For diagrams $A$ and $B$ in Figure \ref{fig5} and the $Y_1$ diagrams obtained from them by moving the inner point to the vertex, call the point moving with the arrow ``point 1'' and the point moving between the other two dots ``point 2.'' We use uniform linear motion for point 1, while point 2 moves uniformly from the black dot to the vertex, and then uniformly (at a usually different rate) from the vertex to the white dot in such a way that for Diagram A (resp.~B) it arrives at the vertex after (resp.~before) the first dot.

Let $d_1$ (resp.~$d_3$) denote the distance from the vertex of the black (resp.~white) dot involved in the uniform motion with the arrow, and $d_2$ (resp.~$d_4$) the distance from the vertex of the other black (resp.~white) dot.
Noting that  point 1 hits the vertex when $t=d_1/(d_1+d_3)$, we choose to have point 2 hit the vertex when $t_0=\max(2d_1/(2d_1+d_3),d_2/(d_2+d_4))$
in Diagram A, and when $t_0=\min(d_1/(d_1+2d_3),d_2/(d_2+d_4))$ in Diagram B. This gives the appropriate values of $t_0$ if $d_4=0$ or $d_2=0$, implying uniform linear motion on the $Y_1$ diagrams. It is important to note, too, that the limiting motion as $d_1$ (resp.~$d_3$) approaches 0 in Diagram A (resp.~B) is the simultaneous uniform linear motion of the $I$ diagram in $E_2'$.

Since most of the $I$ diagrams are in $E_1'$, while most of the $Y$ diagrams are in $E_2'$, only a small amount of checking of compatibility of the motion in an $I$ diagram that is a limit of $Y$ diagrams in the same $E_i'$ is required, and this is easily done. The only $I$ diagram that is a limit of $Y$ diagrams in $E_1'$ is in $E_2'$. Each of the first three $E_2'$ diagrams in Figure \ref{fig5} can approach an $I$ diagram with an impossible motion of moving from an outer black dot to an outer white dot, but the arrangement of black and white in these is opposite to the $I$ diagram in Figure \ref{fig5}, so that limiting diagram is in $E_1'$.
\end{proof}

\begin{proof}[Proof of Theorem \ref{GCthmun}] The lower bounds on $\GC$ follow from the known values of $\TC$.  The upper bounds follow from Propositions \ref{3setsun} and \ref{Yun}.
\end{proof}

%%%%%%CONCLUSION%

%\section{Conclusion}

%We have seen in our main results that in a number of cases in which $G$ is a tree and a two-point configuration space of $G$ is given a natural metric, the geodesic complexity coincides with the topological complexity.  The $\TC$ is known for configuration spaces of a number of graphs, and we wonder whether there are natural examples for which the $\GC$ and $\TC$ differ.  The main difficulty in extending our results to more complicated graphs or higher-order configuration spaces is that the large number of sets on which one needs to construct GMPRs and determine compatibility.

%%%%%BIBLIOGRAPHY
\bibliographystyle{plain}
%\bibliography{bib}{}

\begin{thebibliography}{99}
\bibitem{BridsonHaefliger} M.R.Bridson and A.Haefliger, Metric Spaces of Non-Positive Curvature, Grundlehren der
mathematischen Wissenschaften Vol.\ 319, Springer-Verlag (1999).
\bibitem{Davis} D.M.Davis, {\em An n-dimensional Klein bottle}, Proc. Roy. Soc. Edinburgh Sect. A {\bf 149} (2019) 1207--1221.
\bibitem{D} D.M.Davis, {\em Geodesics in the configuration spaces of two points in $\R^n$},  {\tt arXiv}:2001.00850.
\bibitem{DR} D.M.Davis and D.Recio-Mitter, {\em Geodesic complexity of $n$-dimensional Klein bottles}, {\tt arXiv}:1912.07411.
    \bibitem{Farber} M.Farber, {\em Topological complexity of motion planning}, Discr.~Comp.~Geom. {\bf29} (2003) 211--221.
    \bibitem{Far2} M.Farber, {\em Collision-free motion planning on graphs}, in Algorithmic Foundations of Robotics VI, Springer (2005) 123--128.
\bibitem{G} R.Ghrist, {\em Configuration spaces and braid groups on graphs in robotics}, AMS/IP Studies Advanced Math {\bf 24} (2001) 29--40.
\bibitem{RM} D.Recio-Mitter, {\em Geodesic complexity of motion planning},  {\tt arXiv}:2002.07693.
\end{thebibliography}

\end{document}